\documentclass[12pt]{amsart}
\usepackage[english]{babel}
\usepackage[applemac]{inputenc}
\usepackage[T1]{fontenc}
\usepackage{bbm}
\usepackage{float, verbatim}
\usepackage{amsmath}
\usepackage{amssymb}
\usepackage{amsthm}
\usepackage{amsfonts}
\usepackage{graphicx}
\usepackage{mathtools}
\usepackage[pagebackref]{hyperref}
\usepackage{color}
\usepackage{tikz}
\usepackage{xcolor}
\usepackage{caption}
\usepackage{subcaption}
\usepackage{enumerate}
\usepackage{url}

\hypersetup{
	colorlinks=true,
	linkcolor=blue,
	filecolor=magenta,      
	urlcolor=cyan,
	citecolor= red,
}

\usepackage[shortlabels]{enumitem}

\usepackage[final]{showlabels}

\newcommand{\Haus}{\dim_{\mathrm{H}}}

\newtheorem*{thm*}{Theorem}
\newtheorem*{conj*}{Conjecture}
\newtheorem*{ques*}{Question}
\newtheorem*{rem*}{Remark}
\newtheorem*{defn*}{Definition}
\newtheorem*{mainques*}{Main questions}

\newtheorem{thm}{Theorem}[section]

\newtheorem{lma}[thm]{Lemma}

\newtheorem{rem}[thm]{Remark}

\def\RR{\mathbb{R}}

\def\QQ{\mathbb{Q}}
\def\ZZ{\mathbb{Z}}
\def\NN{\mathbb{N}}

\def\supp{\mathrm{supp}}

\newcommand{\cbr}[1]{\left\{ {#1} \right\}}

\def\one{\mathbf{1}}

\begin{document}
\title{Non-Salem sets in metric Diophantine approximation}


\author{Kyle Hambrook}
\address{Kyle Hambrook\\
	Department of Mathematics and Statistics\\San Jose State University\\One Washington Square\\CA 95192\\USA }
\curraddr{}
\email{kyle.hambrook@sjsu.edu}
\thanks{}

\author{Han Yu}
\address{Han Yu\\
	Department of Pure Mathematics and Mathematical Statistics\\University of Cambridge\\CB3 0WB \\ UK }
\curraddr{}
\email{hy351@cam.ac.uk}
\thanks{}



\begin{abstract}
A classical result of Kaufman states that, for each $\tau>1,$ 
the set of well approximable numbers 
\[
E(\tau)=\{x\in\mathbb{R}: \|qx\| < |q|^{-\tau} \text{ for infinitely many integers q}\}
\] 
is a Salem set with Hausdorff dimension $2/(1+\tau)$.  
A natural question to ask is whether the same phenomena 
holds for well approximable vectors in $\mathbb{R}^n.$ 
We prove that this is in general not the case. 
In addition, we also show that in $\mathbb{R}^n, n\geq 2,$ the set of badly approximable vectors is not Salem.
\end{abstract}

\maketitle

\maketitle
\allowdisplaybreaks
\section{Introduction}
In this paper,  we consider a problem on Fourier dimensions of various sets from metric Diophantine approximation.  
First, we introduce the notion of Fourier dimension.

\begin{defn*}
The Fourier dimension of a Borel probability measure $\mu$ on $\RR^n$ is 
	$$
	\dim_F \mu 
	= 
	\sup \cbr{ 
		s \in [0,n] : 
		|\widehat{\mu}(\xi)| \ll |\xi|^{-s/2} 
	}
	,$$
	where $\hat{\mu}$ is the Fourier transform of $\mu,$ i.e., 
	\[
	\hat{\mu}(\xi)=\int e^{-2\pi i x \cdot \xi}d\mu(x) \quad \text{for $\xi \in \RR^n$}. 
	\]
The Fourier dimension of a set $E \subset \RR^n$ is 
$$
\dim_F E = \sup \cbr{\dim_F \mu : \mu \in \mathcal{P}(E) }, 
$$
where $\mathcal{P}(E)$ is the set of all Borel probability measures $\mu$ on $\RR^n$ 
for which $\supp(\mu)$ 
is a compact subset of $E$. 
(Recall that $\supp(\mu)$ is the largest closed set $C$ for which $\mu(\RR^n \setminus C) = 0$.)  
\end{defn*}	



\begin{rem*}
	Our definition for Fourier dimension is from Matilla \cite[Chapter 3]{Ma2}. 
	There are other versions of Fourier dimension in the literature, not all of which are equivalent. 
	See the discussion at the end of \cite[Page 40]{Ma2}.
\end{rem*}
Note that, if $E$  is a Borel subset of $\RR^n$
\[
\Haus E \geq \dim_F E
\]
where $\Haus E$ is the Hausdorff dimension of $E.$ 
(In this paper, we do not need the definition of Hausdorff dimension, 
but the interested reader can find more details in \cite{Fa} and \cite{Ma1}.)  
It is an interesting problem to find sets which achieve the equality, i.e. $\Haus E =\dim_F E.$ 
\begin{defn*}
	A Salem set is a Borel subset of $\RR^n$ with $\Haus E=\dim_F E.$ 
\end{defn*}
There is no lack of Salem sets. In fact, "most" randomly constructed sets are Salem, see \cite{K85} \footnote{However, some naturally defined random sets are not Salem. See \cite{F1}.\cite{F2}.}.
On the other hand, it is considered to be a challenging problem to find non-random Salem sets with non-integer dimensions.\footnote{With integer dimensions, it is in general easy to find examples. 
For example, consider hyper-surfaces with non-vanishing curvatures.} 
This was first done by Kaufman in \cite{K81} on $\mathbb{R}$ with the following set. 
\begin{defn*}
For $\tau>0$, the set $E(\tau)$ is defined to be 
\[
E(\tau)=\{ x \in \mathbb{R} : \| qx \| < |q|^{-\tau} \text{ for infinitely many } q \in \ZZ \},
\]
where $\|x\|$ is the distance between the real number $x$ and the set of integers.  
Elements of $E(\tau)$ are called well approximable numbers (or, more precisely, $\tau$-well approximable numbers). 
\end{defn*}
\begin{thm*}[Kaufman]
Let $\tau>1.$ The set $E(\tau) \subset \mathbb{R}$ is Salem. 
\end{thm*}

Kaufman's result hinted an approach of find explicit, non-random Salem sets in $\mathbb{R}^n, n\geq 2$.  A plausible guess is the following construction in metric number theory.
\footnote{In fact, the task of finding Salem sets in $\mathbb{R}^n, n\geq 2$ with any given Hausdorff dimension was only resolved very recently; see \cite{HF19}. For $n=2$, see also \cite{H17}.}
\begin{defn*}
	Let $d, n \geq 1$ be integers and let $\tau>0$.  
	The set $E(\tau,d,n)$ is defined to be 
	\[
	E(\tau,d,n)=\{x\in\mathbb{R}^{nd}: |xq-r|_\infty<|q|_\infty^{-\tau}\text{ for infinitely many } (q,r)\in\mathbb{Z}^d\times\mathbb{Z}^n\},
	\]
	where $x \in \mathbb{R}^{nd}$ is considered to be an $n\times d$ matrix and $| \cdot |_\infty$ is the maximal norm in Euclidean spaces. 
	Elements of $E(\tau,d,n)$ are called well approximable vectors (or, more precisely, $\tau$-well approximable vectors). 
\end{defn*}
It is known (\cite{BD86}) that, when $\tau>d/n$, 
\[
\Haus E(\tau,d,n)=n(d-1)+\frac{n+d}{\tau+1}.
\] 
It is unknown whether any of the sets $E(\tau,d,n),\tau>d/n$ is Salem. From Kaufman's result, it is tempting to believe that these sets should be Salem. 
\begin{ques*}
	Let $d,n\geq 1$ be integers and $\tau>d/n.$ Is $E(\tau,d,n)$ a Salem set?
\end{ques*}
This question was raised (in a more general form) in \cite{H19}. 
In the present paper, we answer this question. 
\begin{thm}\label{main1}
		Let $d,n\geq 1$ be integers. Then for each $\tau>d/n,$
		\[
		\dim_F E(\tau,d,n)=\frac{2d}{1+\tau}.
		\]
		In particular, $E(\tau,d,n)$ is not Salem unless $n=d=1.$
\end{thm}
\begin{rem}
For Borel sets, Fourier dimension is in general trickier to deal with than Hausdorff dimension. 
For example, it is in general not easy to confirm that a suspected non-Salem set is actually not Salem. 
To do this, one needs to check the Fourier decay property for all Borel probability measures supported in the set! 
Luckily, this turns out to be not too hard for our problem. 
\end{rem}   

We can also prove a result on Fourier dimension of sets badly approximable vectors. 
\begin{defn*}
	Let $n \geq 1$ be an integer. 
	The set of badly approximable vectors in $\RR^n$ is 
	\[
	Bad(n)=\{x\in\mathbb{R}^n: \exists c(x) > 0, \forall q \in \NN, \|q x\| \geq c(x) q^{-1/n}\}.
	\]
	For $\epsilon>0$, let $Bad(n,\epsilon)$ be 
	\[
	Bad(n,\epsilon)=\{x \in \mathbb{R}^n: \forall q \in \NN,  \|q x\| \geq \epsilon q^{-1/n}\}.
	\] 
\end{defn*}
It is known that the set of badly approximable vectors has full Hausdorff dimension, i.e., $\Haus Bad(n)=n$ for all $n\geq 1.$ 
It is also known that $\dim_F Bad(1)>0.$ 
In fact, $Bad(1,\epsilon)>0$ for all small enough $\epsilon$. 
See for example \cite{K81b}, \cite{JS} and \cite{SS}. 
In this paper, we will prove the following theorem.
\begin{thm}\label{thm:bad}
	Let $n\geq 1$ be an integer. Then we have
	\[
	\dim_F Bad(n) \leq \frac{2n}{n+1}.
	\]
\end{thm}
\begin{rem}
This theorem tells us nothing non-trivial when $n=1.$ However, as long as $n\geq 2,$ we conclude that the set $Bad(n)$ is not Salem. 
\end{rem}

\section{Notations and Preliminaries}
We use the standard Vinogradov symbols $\ll,\gg,\asymp$ as well as the Bachman-Landau notations $O,o.$ 

For the proof of Theorem \ref{main1}, we need the following result of the first named author (\cite{H19}), which proves one half of Theorem \ref{main1}.
\begin{thm}\label{lower}
		Let $d,n\geq 1$ be integers. For each $\tau>d/n,$
	\[
	\dim_F E(\tau,d,n)\geq\frac{2d}{1+\tau}.
	\]
\end{thm}
The proof of the above result in \cite{H19} is constructive and 
provides an explicit probability measure supported by $E(\tau,d,n)$ 
with the demanded Fourier decay property. 
To complete the proof of Theorem \ref{main1}, our task is to show 
that it is not possible to find other measures supported by $E(\tau,d,n)$ 
with faster Fourier decay rate. 
Likewise, to prove Theorem \ref{thm:bad}, we must show that there 
is no probability measure supported on (a compact subset of) $Bad(n)$ 
that satisfies $|\widehat{\mu}(\xi)| \ll |\xi|^{-s/2}$ with 
$s > \frac{2n}{n+1}$. 
To do so, we will use the two theorems below, which count points near lattices. 

For all $\delta\in (0,1)$, $Q \in\mathbb{N}$, 
define 
\[
A(\delta,Q)=\{x\in\mathbb{R}^n: \| Qx  \| < \delta\},
\]
where $\|x\|$ is the Euclidean distance from $x \in \mathbb{R}^n$ to the nearest point of the integer lattice $\mathbb{Z}^n$. 

The following theorem is due to the second author \cite[Theorem 4.1]{Y21}. 
\begin{thm}\label{Lattice Counting}
	Let $n \geq 1$ be an integer. 
	 Let $\delta \in (0,1),$ $Q \in\mathbb{N}.$ 
	 %
	 %
	 Let $\mu$ be a Borel probability measure on $[0,1]^n.$
	 Then 
	\begin{align*}
	\mu(A(\delta,Q)) \ll \delta^n\left( 1+O_{Q\to\infty}(\sum_{\xi \in \ZZ^n, Q|\xi, 0< |\xi| \leq 2Q/\delta} |\hat{\mu}(\xi)|)\right)
	\end{align*}
	and, for all $K > 0$ and $N \in \NN$, 
	\begin{align*}
	\mu(A(\delta,Q)) \gg \delta^n \left(1+O_{Q\to\infty}(\sum_{\xi \in \ZZ^n, Q|\xi, 0< |\xi| \leq KQ/\delta}  |\hat{\mu}(\xi)|)\right)+O_{K\to\infty}(K^{-N}).
	\end{align*}
	Moreover, the implied constant in the $O_{K\to\infty}(.)$ term depends on $N$ only, and all other implied constants are absolute. 
\end{thm}
Theorem \ref{Lattice Counting} will be sufficient to prove Theorem \ref{thm:bad} and to prove Theorem \ref{main1} when $d=1$. 
To prove Theorem \ref{main1} in full generality, 
we will need the following analog of Theorem \ref{Lattice Counting} for linear forms. 

For all non-zero $q=(q_1,\dots,q_d) \in \ZZ^d$ and $\delta>0$, we define 
	\[
	L_{\delta,q}=\{x \in \mathbb{R}^d: \exists r\in\mathbb{Z},|q_1x_1+\dots+q_dx_d-r|<\delta\} 
	\]
	and the $n$-fold product 
	$$
	L_{\delta,q}^n = L_{\delta,q} \times \cdots \times L_{\delta,q} \subset \RR^{nd}. 
	$$

\begin{thm}\label{thm:linearform}
	Let $d,n \geq 1$ be integers. Let $q=(q_1,\dots,q_d) \in \ZZ^d$, $q \neq 0$.  Let $\delta>0$.  
	Let $\mu$ be a Borel probability measure on $[0,1]^{nd}.$ 
	Then 
	\[
	\mu(L_{\delta,q}^n) \ll \delta^n \left(1+  O(\sum_{t \in \ZZ^n, 0 < |t|_{\infty} \leq 2/\delta} |\hat{\mu}(t_1 q,\dots, t_n q)|) \right),
	\]
	and, for all $K > 0$ and $N \in \NN$, 
	\[
	\mu(L_{\delta,q}^n) \gg \delta^n \left(1  +  O(\sum_{t \in \ZZ^n, 0< |t|_{\infty} \leq K/\delta} |\hat{\mu}(t_1q,\dots,t_d q)|) \right)+O_{K\to\infty}(K^{-N}). 
	\]
	Moreover, the implied constant in the $O_{K\to\infty}(.)$ term depends on $N$ only, and all other implied constants are absolute. 
	Notation: In the sums above, 
	$$
	(t_1 q,\dots, t_n q) = (t_1 q_1, \ldots, t_1 q_d, \ldots, t_n q_1, \ldots, t_n q_d) \in \RR^{nd}. 
	$$
\end{thm}
Note that Theorem \ref{Lattice Counting} is (essentially) the $n=1$ case of Theorem \ref{thm:linearform}. 
It is possible to generalize Theorem \ref{thm:linearform} in various ways, but we do not pursue that here. 
We will prove Theorem \ref{thm:linearform} at the end of this paper.

\section{Well approximable vectors}
In this section, we prove Theorem \ref{main1}. 
We will split the proof 
into two parts. 
In the first part, we use Theorem \ref{Lattice Counting} to prove the Theorem \ref{main1} in the special case where $d=1$.  In this case, the underlying geometric idea is very easy to be picked.
In the second part, we use Theorem \ref{thm:linearform} to prove Theorem \ref{main1} in full generality.  The overall argument is not too much different than the $d=1$ case.

\begin{proof}[Proof of Theorem \ref{main1}, Part 1.] 
	Assume $d=1$ and  $n \geq 1$. 
	Let $\tau>1/n.$ 
	Seeking a contradiction, suppose that there exists a 
	$\mu \in \mathcal{P}(E(\tau,1,n)) $ such that $\dim_F \mu > 2/(1+\tau).$  
	This indicates that 
	\begin{align*}
	|\hat{\mu}(\xi)|\ll |\xi|^{-1/(1+\tau)-\epsilon}
	\end{align*}
	for some $\epsilon>0.$
 The support of $\mu$ is contained in $E(\tau,1,n) \cap [-M,M]^n$ for some $M>0$.  
Since $E(\tau,1,n) = E(\tau,1,n) + k$ for all $k \in \ZZ^n$ and since $\frac{1}{m} E(\tau,1,n) \subset E(\tau,1,n)$ for all $m \in \NN$, 
we can translate and scale $\mu$ to obtain a probability measure having support contained in $E(\tau,1,n) \cap [0,1]^d$ 
and obeying the same Fourier decay asymptotic. 
Thus, without loss of generality, we can assume $\supp(\mu) \subset [0,1]^n$. 	
	Let $Q \geq 1$. Let $\delta \in (0,1)$. 
	We apply Theorem \ref{Lattice Counting} to see that
	\[
	\mu(A(\delta,Q)) \ll \delta^n (1+O_{Q\to\infty}(\sum_{Q|\xi,\xi\neq 0,|\xi|\leq 2Q/\delta} |{\hat{\mu}(\xi)}|)).
	\]
	For the $O_{Q\to\infty}$ term, we have 
	(from the Fourier decay assumption), 
	\[
	\sum_{\xi \in \ZZ^n, Q|\xi, 0 < |\xi|\leq 2Q/\delta} |\hat{\mu}(\xi)|\ll Q^{-1/(1+\tau)-\epsilon} \left(\frac{1}{\delta}\right)^{n-\frac{1}{1+\tau}-\epsilon}.
	\]
	Now we choose $\delta$ to be
	\[
	\delta=\delta_Q= Q^{-\tau(\epsilon)},
	\]
	where $\tau(\epsilon) \in (1/n,\tau)$ is a number that will be determined later. 
	Observe that
	\[
	\delta_Q^n Q^{-1/(1+\tau)-\epsilon} \delta_Q^{-(n-1/(1+\tau)-\epsilon)} 
	= Q^{-\tau(\epsilon)/(1+\tau) - \epsilon \tau(\epsilon) -1/(1+\tau)-\epsilon}.
	\]
	We choose $\tau(\epsilon) \in (1/n,\tau)$ such that the exponent of $Q$ is
	\[
	-\tau(\epsilon)\frac{1}{1+\tau}-\epsilon\tau(\epsilon)-\frac{1}{1+\tau}-\epsilon<-1.
	\]
	This is possible because
	\[
	-\tau\frac{1}{1+\tau}-\epsilon\tau-\frac{1}{1+\tau}-\epsilon=-1-\epsilon-\epsilon\tau<-1.
	\]
	From here we see that 
	\[
	\mu(A(\delta,Q))\ll Q^{-\tau(\epsilon) n}+Q^{	-\tau(\epsilon)\frac{1}{1+\tau}-\epsilon\tau(\epsilon)-\frac{1}{1+\tau}-\epsilon}.
	\]
	Then we see that
	\[
	\sum_{Q \geq 1} \mu(A(\delta_Q,Q))<\infty.
	\]
	By the convergence part of the Borel-Cantelli lemma, we see that $\mu$-a.e point $x$ should satisfy
	\[
	\|Qx\| < Q^{-\tau(\epsilon)}
	\]
	for at most finitely many $Q$. 
	However, as the support of $\mu$ is contained in $E(\tau,1,n)$, we see that $\mu$-a.e. point $x$ is contained in $E(\tau,1,n)$, i.e., $\mu$-a.e point $x$ should satisfy 
	\[
	|Qx-r|_{\infty}<Q^{-\tau}
	\]
	for infinitely many $(Q,r) \in \mathbb{Z} \times \mathbb{Z}^n.$ 
	This implies that
	\[
	\|Qx\|\ll Q^{-\tau}
	\]
	for a sequence of increasing integers $Q\to\infty.$ 
	Then, since $\tau(\epsilon)<\tau$, it follows that 
	\[
	\|Qx\|< Q^{-\tau(\epsilon)}
	\]
	for infinitely many $Q$. 
	From this contradiction, we see that 
	\[
	\dim_F E(\tau,1,n) \leq \frac{2}{1+\tau}.
	\]
	Then, from Theorem \ref{lower}, we see that 
	\[
	\dim_F E(\tau,1,n)=\frac{2}{1+\tau}.
	\]
\end{proof}

\begin{proof}[Proof of Theorem \ref{main1}, part 2]
Assume $n,d \geq 1$ are integers and $\tau > d/n$. 
Seeking a contradiction, suppose there exists a $\mu \in \mathcal{P}(E(\tau,d,n))$ such that 
\[
|\hat{\mu}(\xi)|\ll |\xi|^{-s}
\]
for some $s > d/(1+\tau)$.  
The support of $\mu$ is contained in $E(\tau,d,1) \cap [-M,M]^d$ for some $M>0$.  
Since $E(\tau,d,1) = E(\tau,d,1) + k$ for all $k \in \ZZ^d$ and since $\frac{1}{m} E(\tau,d,1) \subset E(\tau,d,1)$ for all $m \in \NN$, 
we can translate and scale $\mu$ to obtain a probability measure having support contained in $E(\tau,d,1) \cap [0,1]^d$ 
and obeying the same Fourier decay asymptotic. 
Thus, without loss of generality, we can assume $\supp(\mu) \subset [0,1]^d$. 
Let $\delta \in (0,1)$. 
From Theorem \ref{thm:linearform}, we have
\[
\mu( L_{\delta,q}^n ) \ll \delta^n \left(1 +  O(\sum_{ t \in \ZZ^n, 0 < |t|_{\infty} < 2/\delta} |\hat{\mu}(t_1q, \dots, t_n q)| ) \right),
\]
Because of the Fourier decay assumption, the sum in $t$ is  
	\[
	\ll \sum_{ t \in \ZZ^n, 0 < |t|_{\infty} < 2/\delta} |(t_1q,\ldots,t_nq)|_{\infty}^{-s} =  |q|_{\infty}^{-s} \sum_{ t \in \ZZ^n, 0 < |t|_{\infty} < 2/\delta} |t|_{\infty}^{-s} \ll |q|_{\infty}^{-s} \delta^{s-n}. 
	\]
Since $\tau > d/n$ and $s > d/(1+\tau)$, 
we can choose $\tau'$ close enough to $\tau$ that $n\tau > n\tau' > d$ 
and $s(1+\tau) > s(1+\tau') > d$. 
Now set 
$$
\delta = \delta_q =  |q|_{\infty}^{-\tau'}. 
$$  
Then 
$$
\sum_{q \in \ZZ^d} \mu( L_{\delta,q}^n ) 
\ll 
\sum_{q \in \ZZ^d}  (|q|_{\infty}^{-n \tau'} + |q|_{\infty}^{-s(1 +\tau')} 
< 
\infty. 
$$
By the convergence part of Borel-Cantelli lemma, for $\mu$-a.e. $x \in \RR^{nd}$, 
we have $x \in L_{\delta,q}^n$ for at most finitely many $q \in \ZZ^d$. 
But 
$$
E(\tau',d,n) = \cbr{x \in \RR^{nd} : x \in L_{\delta,q}^n \text{ for infintely many } q \in \ZZ^d}
$$
Thus $x \notin E(\tau',d,n)$ for  $\mu$-a.e. $x \in \RR^{nd}$. 
Hence $\mu$ cannot be supported on $E(\tau',d,n).$ 
Since $\tau'<\tau$, we have $E(\tau,d,n) \subseteq E(\tau',d,n)$, 
and so $\mu$ cannot be supported on $E(\tau,d,n)$ either.  
This contradiction shows that 
\[
\dim_F E(\tau,d,n) \leq \frac{2d}{1+\tau}. 
\]
Then, from Theorem \ref{lower}, we see that 
\[
\dim_F E(\tau,d,n) = \frac{2d}{1+\tau}.
\]
\end{proof}

\section{Badly approximable vectors}\label{Bad}
In this section, we prove Theorem \ref{thm:bad}. First, we provide a weaker result.
\begin{lma}\label{main2}
	Let $n\geq 1$ and let $\epsilon>0.$ We have
	\[
	\dim_F Bad(n,\epsilon)\leq \frac{2n}{n+1}.
	\]
\end{lma}
\begin{proof}
	Let $B=Bad(n,\epsilon).$ Suppose that $\dim_F B>2n/(n+1).$ Then there is a $\mu\in\mathcal{P}(B)$ with
	\[
	\dim_{F} \mu>2n/(n+1).
	\]
	Thus we have
	\[
	|\hat{\mu}(k)|\ll |k|^{-n/(n+1)-\epsilon'}
	\]
	for some $\epsilon'>0.$ Then we apply Theorem \ref{Lattice Counting}. This time we use the lower bound. We have
		\[
	\mu(A(\delta,Q))\geq c_2\delta^n \left(1+O_{Q\to\infty}(\sum_{\xi \in \ZZ^n, Q|\xi, 0< |\xi| \leq KQ/\delta}  |\hat{\mu}(\xi)|)\right)+O_{K\to\infty}(K^{-N}),
	\]
	where $\delta,Q,K,N$ are numbers which will be chosen later.  First, we check the $O_{Q\to\infty}$ part as in the proof of Theorem \ref{main1},
	\[
	\sum_{\xi \in \ZZ^n, Q|\xi, 0< |\xi| \leq KQ/\delta}  |{\hat{\mu}(\xi)}|\ll Q^{-n/(n+1)-\epsilon'}\left(\frac{K}{\delta}\right)^{n-\frac{n}{n+1}-\epsilon'}.
	\]
	We now choose $\delta=\delta_Q=Q^{-\frac{1+\epsilon'}{n}}.$ We see that the right-hand side above is
	\[
	K^{n-\frac{n}{n+1}-\epsilon'} Q^{-\epsilon'+\frac{\epsilon' n}{n+1}-\frac{\epsilon'}{n}-\frac{{\epsilon'}^2}{n}}.
	\]
	Now we want to determine $K.$ We want that $K$ is not very large so that the right-hand side above is $\ll Q^{-\rho}$ for some $\rho>0.$ This can be achieved by choosing
	\[
	K=K_Q=Q^{\rho'}
	\]
	with a small number $\rho'=\rho'(\epsilon',n)>0.$ 
	After determining $K=K_Q,$ we now choose a large number $N$ such that
	\[
	K^{-N}=o(\delta^n_Q).
	\]
	Again, this can be achieved by making $N=N(\epsilon',n)$ large enough. Then we see that, for large enough $Q,$
	\[
\mu(A(\delta,Q))\gg \delta^n_Q>0.
	\]
	This implies that $B\cap A(\delta_Q,Q)\neq\emptyset$ for all large enough $Q.$ This means that there are points $x_Q$ in $B$ with
	\begin{align}\label{Q}
	\|Qx_Q\| <  Q^{-\epsilon'-1/n}    
	\end{align}
	for all large enough $Q.$ However, since $x_Q\in B,$ we must have
	\begin{align}\label{QQ}
	\|Qx_Q\|  \geq \epsilon Q^{-1/n}.   
	\end{align}
	We see that \eqref{Q} and \eqref{QQ} are not compatible as long as $Q$ is large enough. This contradiction shows that
	\[
	\dim_F B\leq 2n/(n+1). 
	\]
	This is what we want to prove.
\end{proof}

Now we prove Theorem \ref{thm:bad}.

	\begin{proof}[Proof of Theorem \ref{thm:bad}]
		Let $\mu \in \mathcal{P}(\text{Bad}(n))$ be arbitrary.  
		%
		%
		By Lemma \ref{main2}, it will suffice to find some $\epsilon > 0$ and some $\nu \in \mathcal{P}( \text{Bad}( n , \epsilon))$ 
		such that $\dim_F \mu \leq \dim_F \nu$. 
		Note that $\|z\| \geq c$ if and only if $|z-r| \geq c$ for all $r \in \ZZ$. 
		Thus  
		$$
		\text{Bad}(n,\epsilon) 
		= \bigcap_{q \in \ZZ} \bigcap_{r \in \ZZ^n} \cbr{x \in \RR^n : | qx -r | \geq \epsilon |q|^{-1/n} }.  
		$$
		Hence $\text{Bad}(n,\epsilon)$ is closed. 
		Note that 
		$$
		\text{Bad}(n) = \bigcup_{\epsilon \in \QQ, \epsilon > 0} \text{Bad} (n,\epsilon). 
		$$
		Since $\supp(\mu) \subset \text{Bad}(n)$, we have 
		$$
		\supp(\mu) = \bigcup_{\epsilon \in \QQ, \epsilon > 0} (\supp(\mu) \cap \text{Bad}(n,\epsilon)). 
		$$
		Since $\supp(\mu)$ is a closed subset of the complete metric space $\RR^n$, 
		$\supp(\mu)$ is also a complete metric space. 
		Thus the Baire category theorem implies 
		$\supp(\mu)$ is not the countable union 
		of nowhere dense subsets of $\supp(\mu)$. 
		So there is some 
		positive $\epsilon \in \QQ$ 
		such that 
		the closure of $\supp(\mu) \cap \text{Bad}(n,\epsilon)$ in $\supp(\mu)$ 
		has non-empty interior in $\supp(\mu)$. 
		But, since $\text{Bad}(n,\epsilon)$ is closed in $\RR^n$, 
		$\supp(\mu) \cap \text{Bad}(n,\epsilon)$ is closed in $\supp(\mu)$. 
		So the closure of $\supp(\mu) \cap \text{Bad}(n,\epsilon)$ in $\supp(\mu)$ is just $\supp(\mu) \cap \text{Bad}(n,\epsilon)$ itself. 
		Thus $\supp(\mu) \cap \text{Bad}(n,\epsilon))$ 
		has non-empty interior in $\supp(\mu)$. 
		Therefore there is an open set $U$ in $\RR^n$ such that 
		$$
		\emptyset \neq \supp(\mu) \cap U  
		\subset \supp(\mu) \cap \text{Bad}(n,\epsilon).  
		$$
		Since $\emptyset \neq \supp(\mu) \cap U$ and $U$ is open,  
		the definition of $\supp(\mu)$ implies $\mu(\supp(\mu) \cap U)  > 0$. 
		By writing $U$ as a countable union of closed balls, 
		it follows that there exists a closed ball $B \subset U$ such that 
		$\mu(\supp(\mu) \cap B)  > 0$. 
		Since $\supp(\mu) \cap B$ is compact and contained in the open set $U$, 
		we can choose a non-negative compactly supported 
		$C^{\infty}$ function $f$ such that 
		$f=1$ on $\supp(\mu) \cap B$ 
		and $\supp(f) \subset U$. 
		Thus 
		$$
		0 < \mu(\supp(\mu) \cap B) \leq \int f d\mu. 
		$$
		%
		%
		Define a probability measure $\nu$ by $d\nu = c f d\mu$, where $c = (\int f d\mu)^{-1}$. 
		Then 
		$$
		\supp(\nu) \subset \supp(\mu) \cap U \subset \supp(\mu) \cap \text{Bad}(n,\epsilon) \subset \text{Bad}(n,\epsilon). 
		$$
		So $\nu \in \mathcal{P}( \text{Bad}(n,\epsilon)) )$. 
		It remains to show that $\dim_F \mu \leq \dim_F \nu$. 
		Note that 
		$$
		c^{-1}|\widehat{\nu}(\xi)| 
		= |(\widehat{f} \ast \widehat{\mu})(\xi)| 
		\leq \int_{|\xi - \eta| \geq \frac{1}{2} |\xi|}  |\widehat{f}(\eta)| |\widehat{\mu}(\xi-\eta)| d\eta
		+ 
		\int_{|\eta| \geq \frac{1}{2} |\xi|}  |\widehat{f}(\eta)| |\widehat{\mu}(\xi-\eta)| d\eta. 
		$$
		Note also that, since $f$ is compactly supported and $C^{\infty}$, we have $|\widehat{f}(\eta)| \ll |\eta|^{-N}$ for any fixed $N$.   
		It follows easily that, 
		for each fixed $s \in [0,n]$, 
		if $|\widehat{\mu}(\xi)| \ll |\xi|^{-s/2}$, then $|\widehat{\nu}(\xi)| \ll |\xi|^{-s/2}$. 
		Thus $\dim_F \mu \leq \dim_F \nu$. 
\end{proof}

\begin{rem*}
After noting that $\text{Bad}(n) = \bigcup_{\epsilon \in \QQ, \epsilon > 0} \text{Bad} (n,\epsilon)$ and each $\text{Bad} (n,\epsilon)$ is closed, 
the rest of the proof is essentially a verification that Fourier dimension is countably stable for closed sets. 
\end{rem*}

\section{Proof of Theorem \ref{thm:linearform}}

In this section, we prove Theorem \ref{thm:linearform}. 
The key idea is to compute the Fourier coefficients of linear forms 
(linear subspaces of torus viewed as $[0,1]^d$) together with thin 
neighbourhoods around them. Hinted by Pontryagin duality, 
the result essentially says that Fourier transform of linear forms 
are basically supported on the corresponding 'dual spaces'. 

Let $d,n \geq 1$ be integers. Let $q = (q_1,\ldots,q_d) \in \ZZ^d, q \neq 0$. 
Let $\delta > 0$. We start with some preparation for the main argument. 

Define $\delta_{\ast} = \delta/ |q|$.  
Recall that 
$$
L_{\delta,q} = \cbr{x \in \RR^d : \exists r \in \ZZ : |q_1 x_1 + \cdots q_d x_d - r| < \delta }. 
$$
For each $r \in \RR$, define the plane 
$$
P_{q,r} = \cbr{x \in \RR^d : q_1 x_1 + \cdots q_d x_d - r = 0 },
$$
and the union of planes 
$$
P_{\delta,q,r} = \cbr{x \in \RR^d : |q_1 x_1 + \cdots q_d x_d - r| < \delta }. 
$$
Since the Euclidean distance between any two planes $P_{q,r'}$ and $P_{q,r''}$ is $|r' - r''| / |q|$, 
we see that 
that $P_{\delta,q,r}$ is the $\delta_{\ast}$-thickening of the plane $P_{q,r}$, i.e.,  
$$
P_{\delta,q,r} = P_{q,r} + B_{\delta_{\ast}}(0),
$$ 
where $B_{\delta_{\ast}}(0)$ is the Euclidean metric ball of radius $\delta_{\ast}$ around the origin.
Define 
$$
L_q = \bigcup_{r \in \ZZ} P_{q,r} 
$$
Note $L_q$ is a union of planes that are orthogonal to $q$ and spaced a distance of $1/|q|$ apart. 
Then 
$$
L_{\delta,q} = \bigcup_{r \in \ZZ} P_{\delta,q,r} = L_q + B_{\delta_{\ast}}(0) 
$$
Thus $L_{\delta,q}$ is the union of the $\delta_{\ast}$-thickenings of the planes composing $L_q$. 

In an abuse of notation, we use $P_{q,r}$ to denote 
the surface measure on the plane $P_{q,r}$.  
Similarly, we use $L_q$ to denote the surface measure on the union of planes $L_q$. 
Note 
$$
L_q = \sum_{r \in \ZZ} P_{q,r}. 
$$

If we assume (without loss of generality) that $q_1 \neq 0$, 
then 
the measure $P_{q,r}$ is given by 
\begin{align*}
\int_{\RR^d} f(x) dP_{q,r}(x) 
= \int_{\RR^{d-1}} f(g(x_2,\ldots,x_d)) \frac{1}{|q_1|} |q| 
d x_2 \cdots dx_d 
\end{align*}
where 
$$
g(x_2,\ldots,x_d) = (q_1^{-1}(r - q_2 x_2 - \cdots - q_d x_d,x_2,\ldots,x_d)
$$ 
is the parameterization of $P_{q,r}$.  
%
By a straightforward calculation, we have 
$$L_q([0,1]^d) = |q|.$$  
%
%

Since the measure $L_q$ is (a multiple of) 
the restriction of the $(d-1)$-Hausdorff measure on $\RR^d$ to the set $L_q$, 
we have the following property: 
There are constants $a,b>0$ (independent of $q$) such that 
\begin{align}\label{LqAD}
a\epsilon^{d-1} \leq L_q(B_{\epsilon}(x)) \leq b\epsilon^{d-1} 
\end{align}
for all $0 < \epsilon < 1$ and all $x \in L_q$. 

Since $q \in \ZZ^d$, the sets $L_q$  
and $L_{\delta,q}$ are $\ZZ^d$-periodic, 
i.e., they are invariant under translation by elements of $\ZZ^d$. 
The measure $L_q$  and the indicator function $\one_{L_{\delta,q}}$ are also  $\ZZ^d$-periodic. 

For each $k \in \ZZ^d$, the Fourier coefficients of $L_q$ are given by the following formula: 
\begin{align} 
\label{fourier Lq} 
\widehat{L_q}(k) =
\left\{
\begin{array}{ll}
|q| & \text{if  $k \in \ZZ q$ }    \\
0 & \text{if  $k \notin \ZZ q$ }
\end{array}
\right.
\end{align}
Indeed, 
since $q \cdot x \in \ZZ$ for all $x \in L_q$, we have, for all $t \in \ZZ$,  
$$
\widehat{L_q}(tq) = \int_{[0,1]^d} \exp(-2 \pi i tq \cdot x) dL_q(x) 
= L_q([0,1]^d) = |q|. 
$$ 
This proves the first case. 
For the second case, we make a coordinate change carried by a matrix $A \in SL_d(\ZZ)$ so that 
$q=(q_1,q_2,\ldots,q_d)$ is transformed into $q' = (p,0,\ldots,0)$, 
where $p=\gcd(q_1,\ldots,q_d)$. 
We are thus reduced to verifying that $\widehat{L_{q'}}(k)=0$ for $k \notin \ZZ q'$, which is straightfoward. 


We are now ready for the main argument in the proof. 

Let $\mu$ be any Borel probability measure on $\RR^{nd}$ with support contained in $[0,1]^{nd}.$ 
For $z \in \RR^{nd}$, we write $z = (z^{(1)},\ldots,z^{(n)})$, where $z^{(i)} \in \RR^d$. 
Note that 
\begin{align}\label{indicator}
\mu(L_{\delta,q}^n) = \int_{[0,1]^{nd}} \prod_{i=1}^{n} \one_{L_{\delta,q}}(x^{(i)})  d\mu(x) 
\end{align}

Let $\phi$ be a non-negative Schwartz function on $\RR^d$ with $\widehat{\phi}(0)=1$. 
For $x \in \RR^d$, define 
$
\phi_{\delta_{\ast}}(x) 
= 
\phi(x/\delta_{\ast}).   
$
Note that 
$\widehat{\phi_{\delta_{\ast}}}(\xi) = \delta_{\ast}^d \widehat{\phi}(\delta_{\ast} \xi)$ 
for all $\xi \in \RR^d$. 
Note also that $\phi_{\delta_{\ast}} \ast L_q$ is a smooth $\ZZ^d$-periodic function on $\RR^d$. 

Our ultimate goal is to bound the expression in \eqref{indicator}. 
To do so, 
we will approximate $\one_{L_{\delta,q}}$ 
by $\phi_{\delta_{\ast}} \ast L_q$. 
With this in mind, we first consider the integral in \eqref{indicator} with $\one_{L_{\delta,q}}$ replaced by $\phi_{\delta_{\ast}} \ast L_q$. 

Let $K > 0$ be arbitrary. 
By Parseval's theorem and \eqref{fourier Lq}, we have 
\begin{align}
\label{parseval} 
\int_{[0,1]^{nd}} \prod_{i=1}^{n} \phi_{\delta_{\ast}} \ast L_q (x^{(i)}) d\mu(x)
&= 
\sum_{k \in \ZZ^{nd}} \overline{ \widehat{\mu}(k) } \prod_{i=1}^{n}  \widehat{L_q}(k^{(i)}) \widehat{\phi_{\delta_{\ast}}}(k^{(i)})) 
\\
\notag 
&=
\delta_{\ast}^{nd} |q|^{n} \sum_{t \in \ZZ^n} \overline{ \widehat{\mu}(t_1q,\ldots,t_n q) } \prod_{i=1}^{n} \widehat{\phi}(\delta_{\ast} t_i q) 
\\ 
\notag
&= \delta_{\ast}^{n(d-1)} \delta^{n} (1+ S + T), 
\end{align}
where 
\begin{align*}
S = \sum_{t \in \ZZ^n, \, 0 < |t|_{\infty} < K/\delta } \overline{ \widehat{\mu}(t_1q,\ldots,t_n q) } \prod_{i=1}^{n} \widehat{\phi}(\delta_{\ast} t_i q), \\
T = \sum_{t \in \ZZ^n, \, |t|_{\infty} \geq K/\delta } \overline{ \widehat{\mu}(t_1q,\ldots,t_n q) } \prod_{i=1}^{n} \widehat{\phi}(\delta_{\ast} t_i q).   
\end{align*}

Since $|\widehat{\phi}| \leq \widehat{\phi}(0)=1$, we have 
\begin{align}\label{S bound}
|S| \leq \sum_{t \in \ZZ^n, \, 0 < |t|_{\infty} < K/\delta } |\widehat{\mu}(t_1q,\ldots,t_n q) |. 
\end{align}
Since $|\widehat{\mu}| \leq \widehat{\mu}(0)=1$ 
and since $\widehat{\phi}$ is a Schwartz function, we have, for every $N \in \NN$, 
\begin{align}\label{T bound}
|T| 
&\ll 
\sum_{t \in \ZZ^n, \, |t|_{\infty} \geq K/\delta } \prod_{i=1}^{n} (1+ \delta_{\ast} |q| |t_i| |)^{-N-n} 
\\
\notag 
&\ll 
\sum_{t \in \ZZ^n, \, |t|_{\infty} \geq K/\delta } (\delta_{\ast} |q| |t|_{\infty} )^{-N-n} 
\ll
( \delta_{\ast} |q| )^{-n} K^{-N}. 
\end{align}
The implied constants here depend only on $\phi$ and $N$.

To prove the upper bound in Theorem \ref{thm:linearform}, we need to make a more specific choice for $\phi$. 
To prove the lower bound, we need to make a different specific choice for $\phi$. 

We start with the upper bound in Theorem \ref{thm:linearform}. 
We choose $\phi$ as above, but with the additional properties that 
$m(\phi) := \min\cbr{\phi(x) : x \in B_2(0)}$ is positive  
and $\widehat{\phi}=0$ outside $B_K(0)$. 
For all $x \in L_{\delta,q} = L_{q} + B_{\delta_{\ast}}$, there exists $z \in L_q$ with $|x-z|_2 < \delta_{\ast}$, and hence 
$$
\phi_{\delta_{\ast}}(x-y) \geq m(\phi) \one_{B_{2\delta_{\ast}}(x)}(y) \geq m(\phi) \one_{B_{\delta_{\ast}}(z)}(y) 
$$
for all $y \in \RR^d$. 
From this and \eqref{LqAD}, it follows that 
\begin{align}\label{approx upper}
\phi' \ast L_q(x) \geq m(\phi) a \delta_{\ast}^{d-1} \one_{L_{\delta,q}}(x)
\end{align}
for all $x \in \RR^d$. 
Since $m(\phi) > 0$, combining this with \eqref{indicator} and \eqref{parseval} yields 
$$
\mu(L_{\delta,q}^n) \leq (m(\phi) a)^{-n} \delta^{n}(1+S+T). 
$$
Since $\widehat{\phi}=0$ outside $B_K(0)$, we have $T = 0$.  
Setting $K = 2$ and appealing to \eqref{S bound} completes the proof of the upper bound in Theorem \ref{thm:linearform}. 

Now we prove the lower bound in Theorem \ref{thm:linearform}. 
As above, we choose $\phi$ to be a non-negative Schwartz function on $\RR^d$ with $\widehat{\phi}(0)=1$. 
But now we require that $\phi = 0$ outside $B_1(0)$.  
For each $x$ in 
$L_{\delta,q} = L_q + B_{\delta_{\ast}}(0)$,  
there is a $z \in L$ such that $|x-z|_2 < \delta_{\ast}$, 
and hence, for each $y \in L_q$, 
$$
\phi_{\delta_{\ast}}(x-y) 
\leq 
|\phi|_{\infty} \int 1_{B_{\delta_{\ast}}(x)}(y) 
\leq 
|\phi|_{\infty} \int 1_{B_{2\delta_{\ast}}(z)}(y). 
$$
From this and \eqref{LqAD}, it follows that, for all $x$ in 
$L_{\delta,q} = L_q + B_{\delta_{\ast}}(0)$, we have 
$$
\phi_{\delta_{\ast}} \ast L_q(x) \leq |\phi |_{\infty} b (2\delta_{\ast})^{d-1}.  
$$ 
On the other hand, for all $x$ not in 
$L_{\delta,q} = L_q + B_{\delta_{\ast}}(0)$, 
we have 
$\phi_{\delta_{\ast}}(x-y) = 0$ for all $y \in L_q$, 
and so $\phi_{\delta_{\ast}} \ast L_q(x) = 0$. 
Therefore 
\begin{align*}\label{phi lower bound 1}
\phi_{\delta_{\ast}} \ast L_q(x) \leq |\phi|_{\infty} b (2 \delta_{\ast} )^{d-1} \one_{L_{\delta,q}}(x)
\end{align*}
for all $x \in \RR^d$. 
Combining this with \eqref{indicator} and \eqref{parseval} yields 
$$
\mu(L_{\delta,q}^n) \geq 2^{-n(d-1)} (|\phi |_{\infty} b)^{-n} \delta^{n} (1+S+T). 
$$
Applying both \eqref{S bound} and \eqref{T bound} yields the 
lower bound in Theorem \ref{thm:linearform}.

\section{Acknowledgement}
Han Yu was financially supported by the University of Cambridge and the Corpus Christi College, Cambridge. Han Yu has received funding from the European Research Council (ERC) under the European Union’s Horizon 2020 research and innovation programme (grant agreement No. 803711). 


\bibliographystyle{amsplain}

\begin{thebibliography}{99}









\bibitem{BD86} J. D. Bovey and M. M. Dodson, \emph{The Hausdorff dimension of systems of linear forms}, Acta
Arith., \textbf{45}(4), (1986), 337Ñ1§7358.




\bibitem{Fa} K. Falconer, \emph{Fractal geometry: Mathematical foundations and applications, second edition}, John Wiley and Sons, Ltd, (2005).


\bibitem{F1} J. Fraser, T. Orponen and T. Sahlsten, \emph{On Fourier analytic properties of graphs}, International Mathematics Research Notices , (2014), 2730-2745.

\bibitem{F2} J. Fraser and T. Sahlsten, \emph{On the Fourier analytic structure of the Brownian graph},  Analysis and PDE, \emph{11}, (2018), 115-132.

\bibitem{HF19} R. Fraser and K. Hambrook, \emph{Explicit Salem sets in $\mathbb{R}^n$}, preprint, arxiv:1909.04581, (2020).

\bibitem{H17} K. Hambrook,  \emph{Explicit Salem sets in $\mathbb{R}^2$}, Adv. Math., 311, (2017), 634Ñ1§7648

\bibitem{H19} K. Hambrook, \emph{Explicit Salem sets and applications to metrical Diophantine approximation}
Trans. Amer. Math. Soc., \textbf{371}(6), (2019), 4353Ñ1§74376.


\bibitem{JS} T. Jordan and T. Sahlsten, \emph{Fourier transforms of Gibbs measures for the Gauss map},
Math. Ann., \textbf{364} (3-4), 983-1023, (2016).





\bibitem{K85} J.-P. Kahane, \emph{Some random series of functions}, Cambridge Studies in Advanced
Mathematics. Cambridge University Press, Cambridge, second edition, 1985.

\bibitem{K81} R. Kaufman, \emph{On the theorem of Jarn\'{ı}k and Besicovitch}, Acta Arith., \textbf{39}(3), (1981), 265Ñ1§7267.

\bibitem{K81b} R. Kaufman, \emph{Continued fractions and Fourier transforms}, Mathematika, \textbf{27}, (1981), 262Ñ1§7267.










\bibitem{Ma1} P. Mattila, \emph{Geometry of sets and measures in Euclidean spaces: Fractals and rectifiability}, Cambridge Studies in Advanced Mathematics, Cambridge University Press, (1999).

\bibitem{Ma2} P. Mattila, \emph{Fourier Analysis and Hausdorff Dimension}, Cambridge Studies in Advanced Mathematics, Cambridge University Press, (2015).









\bibitem{SS} T. Sahlsten and C. Stevens, \emph{Fourier transform and expanding maps on Cantor sets}, preprint, arxiv:2009.01703, (2020).









\bibitem{Y21} H. Yu, \emph{Rational points near self-similar sets}, preprint,  arXiv:2101.05910, (2021).

\end{thebibliography}

\end{document}